\documentclass{article}
\usepackage{amsthm}
\usepackage{amsmath,amssymb,latexsym}
\usepackage[round]{natbib}
\usepackage{rotate}
\usepackage{enumerate}
\usepackage{graphicx}
\usepackage{float}
\usepackage{ragged2e}
\usepackage{tikz}
\usepackage{color}
\usetikzlibrary{shapes}
\usetikzlibrary{plotmarks}
\usetikzlibrary{arrows}
\usetikzlibrary{positioning}
\usepackage{tkz-graph}
\usepackage{authblk}

\title{Weak G\"odel's incompleteness property for some decidable versions of the calculus of relations}
\author{MOHAMED KHALED}
\affil{Department of Mathematics and its applications, Central European University, Budapest, Hungary}
\affil{Department of Mathematics, Faculty of Science, Cairo University, Giza, Egypt}
\date{}
\def\f#1#2{\mathfrak{Fr}_{#1}#2}
\def\t#1#2{\mathfrak{Tm}_{#1}#2}
\def\c{\tau,H}

\def\a#1{\mathfrak{#1}}
\def\B{\mathfrak{B}}
\def\qut#1{``#1''}
\def\G{\a{G}^{\c}}
\newtheorem{definition}{Defintion}[section]
\newtheorem{thm}{Theorem}[section]
\newtheorem{prop}{Ptoposition}[section]
\newtheorem{remark}{Remark}[section]
\begin{document}

\maketitle

\begin{abstract}Relativization is one of the central topics in the study of algebras
of relations. Some relativized relation algebras behave much nicer
than the original relation algebras. In this paper, we study the
atomicity of the finitely generated free algebras of these nice
classes of relativized relation algebras. In particular, we give an
answer for the open problem, posed by I. N\'emeti in 1985, which
asks whether the finitely generated free algebras of the class of
the weak associative relation algebras $WA$ are atomic or not.
\end{abstract}

\section{Introduction}
The first order predicate calculus has its origins in the calculus
of relations. The calculus of relations was created and developed in
the second half of the nineteenth century by De Morgan, Peirce and
Schr\"oder. The creation of this calculus was the result of the
continuous efforts searching for a \qut{good general algebra of
logic}. But these efforts took place decades before the emergence of
first order calculus. The early notation for quantifiers originates
with Peirce. L\"owenheim's original version of the
L\"owenheim-Skolem theorem, c.f. \citep{low15}, is not a theorem
about first order logic but about the calculus of relations.

In 1940, Tarski proposed an axiomatization for a large part of the
calculus of relations. In the next decades Tarski's axiomatization
led to the creation of the theory of relation algebras which was
shown to be incomplete by Lyndon's discovery of nonrepresentable
relation algebras. Alfred Tarski showed that mathematics can be
built up in the equational theory of the relation algebras, hence it
is undecidable \citep{tg}. He raised the problem \qut{how much
associativity of relation composition is needed for this result}.
Roger Maddux defined the class of weakly-associative relation
algebra, $WA$, by weakening the associativity of the relation
composition. Istv\'an N\'emeti showed that the equational theory of
$WA$ is decidable.

This class also can be seen as a class of relativized relation
algebras. Indeed, it is proved in \citep{mad82} that an algebra is a
$WA$ if and only if it is isomorphic to the concrete algebra of some
subrelations of a symmetric and reflexive relation. The notion of a
relativized algebra has been introduced in the theory of Boolean
algebras and then it was extended to algebras of logics by Leon
Henkin. Relativization in algebraic logic started as a technique for
generalizing representations of algebras of logics. Relativization
of an algebra amounts to intersecting all its elements with a fixed
set (usually a subset of the unit) and to defining the new
operations as the restrictions of the old operations on this set.
Relativized algebras were not really studied in their own right, but
as tools to obtain results for the standard algebras. At the end of
the twentieth century, Andr\'eka, van Benthem, Monk and N\'emeti
started promoting relativized algebras as structures which are
interesting independently of their classical versions, see e.g.
\citep{ca3}. Indeed, relativization in many cases turns the negative
results into positive ones. Several relativized versions of algebras
do have most of the nice properties which their standard
counterparts lack.

From the universal algebra, the free algebras of a variety play an
essential role in understanding this variety. They show, in some
sense, the structure of the different \qut{concepts} (represented by
terms) of the variety. The intrinsic structures of the finitely
generated free relativized relation algebras are very involved. Some
problems concerning these algebras are still open. For example, the
problem addressing the atomicity of these algebras has not been
solved yet. The non-atomicity of the free algebras of logics is
equivalent to weak G\"odel's incompleteness property of the
corresponding logic. See \citep[proposition 8]{prenem85} and
\citep{gyn}.

In the present paper, we study the atomicity of the free algebras of
some interesting classes of relativized relation algebras. In
particular, we give an answer for the atomicity problem of the free
algebras of the class $WA$. This problem goes back to 1985 when I.
N\'emeti posed it in his Academic Doctoral Dissertation
\citep{nem86}. In 1991, N\'emeti posed the same problem again in
\citep[Problem 38]{al}. This problem was posed again as an open
problem in 2013 in the most recent book in algebraic logic
\citep[Problem 1.3.3]{ca4}. Solving this problem, we show that the
free algebras of $WA$ generated by at least one generator are not
atomic but the $0$-generated free algebra of $WA$ is finite, hence
atomic.

One of the logics corresponding to relativized relation algebras is called
arrow logic. It is a two-dimensional modal logic and it has various
applications, e.g., in linguistics (dynamic semantics of natural
language, relational semantics of Lambek Calculus),
 and in computer science (dynamic algebra, dynamic logic). For more
about relativized relation algebras and arrow logic as modal logic
see \citep{vB91, vBe94, vB96, GKWZ03, HH02, inves, mv, miku93}.

We need to recall some notions from universal algebra. For the
definitions of these notions, one can look up any book in universal
algebra (e.g., \citep{bua}). Let $t$ be any algebraic type and let
$K$ be a class of algebras of type $t$. Let $m$ be any cardinal.
$\t{m,t}{}$ denotes the term algebra of type $t$ generated by
$m$-many free variables. $\f{m}{K}$ denotes the free algebra of the
class $K$ generated by $m$-many generators.

\section{Interesting relativized relation algebras}Let $W$ be an arbitrary set of ordered pairs. Set $Id^{[W]}=\{(r,s)\in W: r=s\}$.
Define a unary operation $\Breve{\text{ }}^{[W]}$ and a binary operation $;^{[W]}$ on $\mathcal{P}(W)$ as follows. For any $R,S\subseteq W$,
$${\Breve{R}^{[W]}}=\{(r,s)\in W:(s,r)\in R\},$$
$$R;^{[W]}S=\{(r,s)\in W:(\exists u) (r,u)\in R\& (u,s)\in S\}.$$
When no confusion is likely, we merely omit the superscript $[W]$
from the above defined objects. A relativized relation set algebra
is an algebra of the form
$$\mathfrak{A}=\langle A, \cap, \cup, \setminus, \emptyset, W, ;^{[W]},
\Breve{\text{ }}^{[W]}, Id^{[W]}\rangle,$$ where $W$ is an arbitrary
set of ordered pairs and $A\subseteq\mathcal{P}(W)$ is a family of
subsets of $W$ that is closed under the Boolean set theoretic
operations $\cap,\cup,\setminus$, closed under the relativized
operations $\Breve{\text{ }}^{[W]},;^{[W]}$ and contains the
following sets $\emptyset, W, Id^{[W]}$. Suppose that $U$ is the
smallest set such that $W\subseteq U\times U$. Then $W$ and $U$ are
called the unit and the base of $\mathfrak{A}$, respectively. The
class of relativized relation set algebras is denoted by
$FRA_{\emptyset}$. Let $\it{rl}$ denote the algebraic type of
$FRA_{\emptyset}$ where $1'$ denotes the constant $Id$ called
identity, its complement is denoted by $0'$ and is called diversity.
\begin{definition}
Let $H\subseteq\{R,S,T\}$, where $R,S$ and $T$ stand for
\qut{Reflexive}, \qut{Symmetric} and \qut{Transitive} respectively.
A relation is said to be an $H$-relation if it satisfies the
properties in $H$. The class of $H$-relativized relation set
algebras, $FRA_H$, is the subclass of $FRA_{\emptyset}$ which
contains all of those algebras whose units are $H$-relations on
their bases, $RRA_H$ denotes the class of the algebras isomorphic to
$H$-relativized representable relation set algebras.
\end{definition}
The notion of $H$-relativization with $H\subseteq\{R,S,T\}$ we use
here was suggested by M.\ Marx. We note that $RRA_{\{R,S,T\}}$ is
the class of the standard representable relation algebras. We also
note that the class of the weak associative relation algebras $WA$
coincides with the class $RRA_{\{R,S\}}$. In \citep{and91},
\citep{and94b}, \citep{and96}, \citep{kramer}, \citep{marx99},
\citep{inves}, \citep{mad82} and \citep{nem87}, it was shown that,
for arbitrary $H\subseteq\{R,S,T\}$, the class $RRA_H$ enjoys any of
the following properties if and only if $T\not\in H$: finite
axiomatizability, decidability, finite algebra property, finite base
property, weak and strong interpolation, Beth definability and super
amalgamation property.

In this paper, we concentrate only on those classes $RRA_H$, where
$H\subseteq\{S,R\}$. Our main aim is to prove that, for every
$m\in\omega$ and every $H\subseteq\{R,S\}$, the $m$-generated free
algebra of the class $RRA_H$ is atomic if and only if $m=0$ and
$H=\{S,R\}$. We concentrate on finite numbers $m$ only because it is
known that all the infinitely-generated free $RRA_H$ algebras are
atomless.
\section{Normal forms for relation type}
In this section, we give a disjunctive normal form for any term in
the signature of relation algebras. Disjunctive normal forms can
provide elegant and constructive proofs of many standard results,
c.f., \citep{anderson} and \citep{fine}.

Throughout this section, we fix $m\in\omega$. For every $n$, we
define a set $F_n^m\subseteq\t{m,rl}{}$ of normal forms of degree $n$
such that every member of $F_n^m$ contains complete
information about the normal forms of the smaller degrees. Then, we
write up each term in $\t{m,rl}{}$ as a
disjunction of normal forms of the same degree.

We need to set up some notation and definitions.
Let $n\in\omega$ and let $\tau_0,\ldots,\tau_{n}\in\t{m,rl}{}$.
Define $\prod_{i=0}^{0}\tau_i:=\tau_0$ and inductively define
$\prod_{i=0}^{k}\tau_i:=(\prod_{i=0}^{k-1}\tau_i)\cdot\tau_k$, for
every $1\leq k\leq n$. Let $Y\subseteq\t{m,rl}{}$ be a finite set of
terms. If $Y=\emptyset$, then define $\prod Y=1$. Assume that $\mid
Y\mid= n$, for some $n\geq 1$. Pick any bijection
$f:n\rightarrowtail\!\!\!\!\!\rightarrow Y$ and let $$\prod
Y:={\prod}^{(f)}Y:=\prod\limits_{i=0}^{n-1}f(i).$$ This is ambiguous,
but in this paper we deal with classes of Boolean algebras with
operators, so it doesn't matter which bijection is taken for the
above product.
\begin{definition}
Let $\emptyset\not=Y\subseteq\t{m,rl}{}$ be finite and
let $\alpha\in{{^Y}\{-1,1\}}$. Define
\begin{enumerate}
\item $CY=\{x;y:x,y\in Y\}\cup\{\Breve{x}:x\in Y\}$, the one-step closure of $Y$ by
the modal operations $;$ and $\Breve{  }$.
\item $Y^{\alpha}=\prod\{x^{\alpha}:x\in Y\}$,
where for every $x\in Y$, $x^{\alpha}=x$ if $\alpha(x)=1$ and $x^{\alpha}=-x$ if $\alpha(x)=-1$.
\end{enumerate}
\end{definition}
\begin{definition}Let $D_m=\{1^{'},x_0,\ldots,x_{m-1}\}$, where $x_0,\ldots,x_{m-1}$ are the $m$ free
variables that generate $\t{m,rl}{}$. For every $n\in\omega$, we
define the followings inductively.
\begin{enumerate}[-]
\item The normal forms of degree $0$, $F_{0}^m=\{D_m^{\beta}:\beta\in{^{D_m}\{-1,1\}}\}$.
\item The set of normal forms of degree $n+1$, $$F_{n+1}^m=\{D_m^{\beta}\cdot (CF_n^m)^{\alpha}:\beta\in{^{D_m}\{-1,1\}}\text{ and }\alpha\in{^{CF_n^m}\{-1,1\}}\}.$$
\item The set of all forms, $F^m=\bigcup_{k\in\omega} F_k^m$.
\end{enumerate}
\end{definition}
Every form in $F_0^m$ is determined by the information telling
whether it is below the identity or the diversity and whether it is
below any free variable or its complement. Every form of degree
$n+1$, $n\in\omega$, is determined by the same information plus
information telling whether this term is below (or below the
complement of) the composition of any couple of forms in $F_{n}^m$,
and whether this form is below (or below the complement of) the
converse of any form in $F_{n}^m$.

Let $K$ be the class of all Boolean algebras with operators of type
$rl$. The next theorem says that, for every $n\in\omega$, the
normal forms of degree $n$ form a partition of the unit
(inside $K$). It also indicates that every term in $\t{m,rl}{}$ can
be rewritten in the form of disjunctions of some normal forms of the
same degree.
\begin{thm}\label{andreka} Let $n\in\omega$. Then the followings are true:
\begin{enumerate}
\renewcommand{\theenumi}{(\roman{enumi})}
\item\label{and1} $K\models\sum F_n^m=1$.
\item\label{and2} For every $\tau,\sigma\in F_n^m$, if $\tau\not=\sigma$ then $K\models\tau\cdot \sigma=0$.
\item\label{and3} There exists an effective method to find, for every $\tau\in F_n^m$, a finite $S\subseteq F_{n+1}^m$ such that $K\models\tau=\sum S$.
\item\label{and4} There exists an effective method to find, for every $\tau\in\t{m,rl}{}$, an $k\in\omega$ and a finite $S_{\tau}\subseteq F^m_k$ such that $K\models\tau=\sum S_{\tau}$.
\end{enumerate}
\end{thm}
\begin{proof}\
\begin{enumerate}
\renewcommand{\theenumi}{(\roman{enumi})}
\item Since the Boolean reduct of every member of $K$ is Boolean algebra, we have $(\forall S\subseteq\t{m,rl}{})K\models \Sigma\{S^{\beta}:\beta\in{^S{\{-1,1\}}}\}=1$.
In particular, for every $n\geq 1$, we should have
\begin{eqnarray*}
K\models\Sigma F_0^m&=&\Sigma\{D_m^{\alpha}:\alpha\in{^{D_m}{\{-1,1\}}}\}=1,\text{ and }\\
K\models \Sigma F_n^m&=&\Sigma\{D_m^{\alpha}\cdot\Sigma\{C(F_{n-1}^{\beta}):\beta\in{^{F_{n-1}^m}\{-1,1\}}\}:\alpha\in{^{D_m}\{-1,1\}}\}\\
&=& \Sigma\{D_m^{\alpha}:\alpha\in{^{D_m}\{-1,1\}}\}\\
&=& 1.
\end{eqnarray*}
\item Let $\alpha_1,\alpha_2\in{^{D_m}\{-1,1\}}$ be such that $\alpha_1\not=\alpha_2$. Then there exists $x\in D_m$ such that, without loss of generality, $\alpha_1(x)=1$ and $\alpha_2(x)=-1$. Therefore,  $K\models D_m^{\alpha_1}\leq x$, $K\models D_m^{\alpha_2}\leq -x$ and $K\models D_m^{\alpha_1}\cdot D_m^{\alpha_2}=0$. Let $n\geq 1$, and let $\tau,\sigma\in F_n^m$. Similarly, if $\tau\not=\sigma$, then without loss of generality we can assume that there exists $y\in D_m\cup CF_{n-1}^m$ such that $K\models\tau\leq y$ and $K\models \sigma\leq-y$.
Therefore, $K\models \tau\cdot\sigma=0$, as desired.
\item For every $\alpha\in{^{D_m}\{-1,1\}}$, we have
$K\models D_m^{\alpha}=\Sigma\{D_m^{\alpha}\cdot(CF_0^m)^{\beta}:\beta\in{^{F_0^m}\{-1,1\}}\}$.
Inductively, let $n\geq 1$ and assume that for every $\sigma\in F_{n-1}^m$ there exists $S_{\sigma}\subseteq F_n^m$ such that $K\models \sigma=\Sigma S_{\sigma}$. For every $\sigma_1,\sigma_2\in F_{n-1}^m$, define $S_{\sigma_1;\sigma_2}=\{\gamma_1;\gamma_2:\gamma_1\in S_{\sigma_1}\text{ and }\gamma_2\in S_{\sigma_2}\}$ and $S_{\Breve{\sigma}_1}=\{\breve{\gamma}:\gamma\in S_{\sigma_1}\}$. Let $\tau=D_m^{\alpha}\cdot(CF_{n-1}^m)^{\beta}\in F_n^m$. For every $\beta^{'}\in{^{CF_n^m}\{-1,1\}}$, we say that $\beta^{'}$ is compatible with $\beta$, in symbols $\beta^{'}\sim\beta$, if for every $\sigma\in CF_{n-1}^m$,
$$\beta(\sigma)=1\iff (\exists \gamma\in S_{\sigma}) \beta^{'}(\gamma)=1.$$
Let $S_{\tau}=\{D_m^{\alpha}\cdot(CF_n^m)^{\beta^{'}}:\beta^{'}\in{^{CF_n^m}\{-1,1\}},\beta^{'}\sim\beta\}\subseteq F_{n+1}^m$. Recall that $K$ is a class of Boolean algebras with operators, therefore $K\models \tau=\Sigma S_{\tau}$.
\item By induction on terms. For every $\tau\in D_m$, we have
$$K\models\tau=\Sigma\{D_m^{\alpha}:\alpha\in{^{D_m}\{-1,1\}},\alpha(\tau)=1\}.$$
Let $\sigma_1,\sigma_2\in\t{m,rl}{}$ be such that there is an effective method to find $n_1,n_2$ and finite $S_1\subseteq F_{n_1}^m$ and $S_2\subseteq F_{n_2}^m$ such that $K\models \sigma_1=\Sigma S_1$ and $K\models \sigma_2=\Sigma S_2$. By item \ref{and3} we may assume that $n_1=n_2=:n$.
\begin{enumerate}[-]
\item If $\tau=\sigma_1+\sigma_2$ then
$K\models\tau=\Sigma (S_1\cup S_2)$.
\item If $\tau=\sigma_1\cdot\sigma_2$ then $K\models\tau=\Sigma\{x\cdot y:x\in S_1,y\in S_2\}$. By item \ref{and2}, it is clear that $\{x\cdot y:x\in S_1,y\in S_2\}\subseteq F_n^m$.
\item If $\tau=-\sigma_1$, then
$K\models\tau=\Sigma (F_n^m\setminus S_1)$.
\item If $\tau=\sigma_1;\sigma_2$ then, for every $w\in S:=\{y;z:y\in S_1, z\in S_2\}$, let $$S_w=\{D_m^{\alpha}\cdot(CF_{n}^m)^{\beta}:\alpha\in{^{D_m}\{-1,1\}},\beta\in{^{F_n^m}\{-1,1\}},\beta(w)=1\}.$$ Therefore, $K\models \tau=\Sigma\bigcup \{S_w:w\in S\}$.
\item If $\tau=\breve{\sigma}_1$ then, for every $w\in S:=\{\breve{y}:y\in S_1\}$, let $$S_w=\{D_m^{\alpha}\cdot(CF_{n}^m)^{\beta}:\alpha\in{^{D_m}\{-1,1\}},\beta\in{^{F_n^m}\{-1,1\}},\beta(w)=1\}.$$ Therefore, $K\models \tau=\Sigma\bigcup \{S_w:w\in S\}$.
\end{enumerate}
\end{enumerate}
\end{proof}

Back to the relativized relation algebras. Let $k\in\omega$. Theorem \ref{andreka}
(\ref{and1}, \ref{and2}) can also be used to label the elements of
the unit of any relativized relation algebra with normal forms from
$F^m_k$. Let $\mathfrak{A}\in RRA_{\emptyset}$ and $ev$ be some
evaluation of $x_0,\ldots,x_{m-1}$ into $A$. Suppose that $V$ is the
unit of $\mathfrak{A}$. Let $(r,s)\in V$, define
$D^{\mathfrak{A},\iota}_k(r,s):=\tau$, where $\tau$ is the unique
form in $F_k^{m}$ such that $(\mathfrak{A},\iota,(r,s))\models\tau$,
where, for every term $\sigma$, we write
$(\mathfrak{A},ev,(r,s))\models\sigma$ if and only if
$(r,s)\in[\sigma]^{\mathfrak{A}}_{v}$.
\begin{remark}\label{fayroz}
Suppose that $k\geq1$. Let $(r,s)\in V$ and suppose that $U$ is the base of $\a{A}$. Note that, in order to determine the normal form in $F^m_k$ that $(r,s)$ satisfies in $(\mathfrak{A},\iota)$, it is
necessary and sufficient to determine the terms from $F_{k-1}^m$ which the neighbors of $(r,s)$ satisfy, where the neighbors of $(r,s)$ is defined as
\begin{eqnarray*}
nghbr(r,s)=&&\{(r,s),(s,r), (r,r), (s,s)\}\cap V\\
&\cup&\{(x,w) :
(w,y)\in V, w\in U, \{ x,y\}\subseteq\{ r,s\}\}\\
&\cup&\{(w,y) :
(x,w)\in V, w\in U, \{ x,y\}\subseteq\{ r,s\}\}.
\end{eqnarray*}
This is so by the definition of normal forms. Indeed, every normal form was determined by the information on the forms of the first smaller degree.
\end{remark} The following definition focuses on how to obtain the information that the normal forms carry from their syntactical construction.
\begin{definition}
Define $color_m:F^m\rightarrow F^m$ as follows. For every $k\in\omega$, every $\beta\in{^{D_m}\{-1,1\}}$ and every $\alpha\in{^{CF_k}\{-1,1\}}$, define $$color_m(D_m^{\beta}\cdot(CF_k)^{\alpha})=color_m(D_m^{\beta})=\{y\in D_m:\beta(y)=1\}.$$
\end{definition}
For every term $\tau\in F^m$, $\tau$ is said to be white if $1^{'}\in color_m(\tau)$ and is said to be black otherwise.
\begin{definition}
Define $sub_m:F^m\rightarrow \mathcal{P}({^2{F_k^m}})$ as
follows:
\begin{enumerate}
\item For every $\tau\in F_0^m$, $sub_m(\tau)=\emptyset$.
\item Let $k\in\omega$, $\beta\in{^{D_m}\{-1,1\}}$, $\alpha\in{^{CF_k^m}\{-1,1\}}$ and $\tau=D_m^{\beta}\cdot (CF_k^m)^{\alpha}\in F_{k+1}^m$. Define $sub_m(\tau)=\{(\sigma_1,\sigma_2)\in {^2{F_k^m}}:\beta(\sigma_1;\sigma_2)=1\}$.
\end{enumerate}
\end{definition}
\begin{definition}
Define the following partial functions $f_m,R_m,L_m:F^m\rightarrow F^m$ as follows:
\begin{enumerate}
\item For every $\tau\in F_0^m$, $\tau\not\in dom(f_m)\cup dom(R_m)\cup dom(L_m)$.
\item Let $k\in\omega$, $\beta\in{^{D_m}\{-1,1\}}$, $\alpha\in{^{CF_k^m}\{-1,1\}}$ and $\tau=D_m^{\beta}\cdot (CF_k^m)^{\alpha}\in F_{k+1}^m$.
\begin{enumerate}
\item If there exists a unique $\sigma\in F_k^m$ such that $\alpha(\Breve{\sigma})=1$ then $\tau\in dom(f_m)$ and $f_m(\tau)=\sigma$. Otherwise, $\tau\not\in dom(f_m)$.
\item If there exists a unique white $\lambda\in F_k^m$ such that $\alpha(\sigma;\lambda)=1$ for some $\sigma\in F_k^m$ then $\tau\in dom(R_m)$ and $R_m(\tau)=\lambda$. Otherwise, $\tau\not\in dom(R_m)$.
\item If there exists a unique white $\lambda\in F_k^m$ such that $\alpha(\lambda;\sigma)=1$ for some $\sigma\in F_k^m$ then $\tau\in dom(L_m)$ and $L_m(\tau)=\lambda$. Otherwise, $\tau\not\in dom(L_m)$.
\end{enumerate}
\end{enumerate}
\end{definition}

\section{The atomicity of the free relativized relation algebras}
Throughout this section, let $m\in\omega$ and
$H\subseteq\{R,S\}$ be arbitrary but fixed. Recall that
we are searching for the atoms in $\f{m}{RRA_H}$. By theorem \ref{andreka} (\ref{and4}), it is
enough to search for the atoms among the normal forms. Our main aim is to prove
the following.
\begin{thm}\label{didit}
$\f{m}{RRA_H}$ is atomic if and only if $m=0$ and
$H=\{R,S\}$.
\end{thm}
We prove some propositions considering some special cases of the
above theorem. The next proposition proves one direction of theorem
\ref{didit}.
\begin{prop}
The free algebra $\f{0}{RRA_{\{R,S\}}}$ is finite, hence atomic.
\end{prop}
\begin{proof}
Let $H=\{R,S\}$ and let $Y=\{e_1,e_2,m_2,m_3\}$, where
\begin{eqnarray*}
e_{1}= 1^{'}\cdot -(0^{'};0^{'}), &\text{ }& e_{2}= 1^{'}\cdot (0^{'};0^{'}),\\
m_{2}= 0^{'}\cdot -(0^{'};0^{'}), &\text{ }& m_{3}= 0^{'}\cdot (0^{'};0^{'}).
\end{eqnarray*}
Clearly, for every $\tau\in F_1^0$ there exists $\sigma\in Y$ such that $\f{0}{RRA_H}\models \tau=\sigma$. It is easy to check the following.
$$\Breve{e_1}=e_1, \Breve{e_2}=e_2, \Breve{m_2}=m_2\text{ and }\Breve{m_3}=m_3.$$
Also, in $\f{0}{RRA_H}$ we have
\begin{center}
\begin{tabular}{ | c | c c c c | }
\hline
\text{ } $;$ \text{ } & \text{ } $e_1$ \text{ } & \text{ } $e_2$ \text{ } & \text{ } $m_2$ \text{ } & \text{ } $m_3$ \text{ } \\
\hline
\text{ } $e_1$ \text{ } & $e_1$ & $0$ & $0$ & $0$ \\
$e_2$ & $0$ & $e_2$ & $m_2$ & $m_3$ \\
$m_2$ & $0$ & $m_2$ & $0$ & $0$\\
$m_3$ & $0$ & $m_3$ & $0$ & \text{ } $e_2+m_3$ \text{ } \\
 \hline
\end{tabular}
\end{center}
Therefore, for every $\tau\in F^0$ there exists $\sigma\in Y$ such
that $\f{0}{RRA_H}\models\tau=\sigma$. Hence, theorem \ref{andreka},
\ref{and4}, implies that $\f{0}{RRA_H}$ is $+$ generated by
$\{e_1,e_2,m_2,m_3\}$. Hence $\f{0}{RRA_H}$ is finite, and
consequently atomic. In fact, $Y$ is the set of atoms of
$\f{0}{RRA_H}$, thus the above give a complete description of this free algebra.
\end{proof}

To prove the other direction of the theorem, suppose that $H\not=\{R,S\}$ or $m>0$. Let $t=0^{'}\cdot\breve{0^{'}}\cdot((0^{'}\cdot\breve{0^{'}});(0^{'}\cdot\breve{0^{'}}))$, we show that there is no atom below $t$ in $\f{m}{RRA_H}$. To this end, fix a finite number $q\geq 1$ and a satisfiable normal form $\tau\in F^m_q$ such that $\f{m}{RRA_H}\models0\not=\tau\leq t$. Our strategy goes through the following steps.
\begin{description}
\item[Step 1] Construct an algebra $\a{G}^{\c}\in RRA_H$ and an evaluation $\iota^{\c}$.
\item[Step 2] Prove that $(\G,\iota^{\c})$ witnesses the satisfiability of the form $\tau$.
\item[Step 3] Select a special sequence of edges $e_q,\ldots,e_0\in E^{\c}$ such that $l^{\c}(e_q)=\tau$.
\item[Step 4] Use this sequence to extend $(\G,\iota^{\c})$ to $(\G_+,\iota^{\c}_+)$.
\item[Step 5] Prove that every element of the above sequence satisfies a term in $(\G,\iota^{\c})$ and another one in $(\G_+,\iota^{\c}_+)$ such that both of them are disjoint and each of which is below its labels.
\end{description}
Hence the labels of the sequence selected in Step 2 are not atoms in
$\f{m}{RRA_H}$. Therefore, $\tau$ is not an atom in $\f{m}{RRA_H}$.

\paragraph{Step 1:} We first construct a graph $G^{\tau}$ that
acts as the unit of $\a{G}^{\tau}$. The construction of $G^{\tau}$
goes through countably many rounds. At round $n\in\omega$, we
construct a graph $G_n=(V_n,E_n)$ with a labeling function
$l_n:E_n\rightarrow\bigcup\{F^m_j:0\leq j\leq q\}$ and with a depth
function $d_n:E_n\cup Id_{V_n}\rightarrow\{0,\ldots,q\}$. To achieve
our purpose, we require our constructed graphs to obey the following
\qut{consistency} conditions. For every $n\in\omega$ and every
$u,v,w\in V_n$ such that $(u,v)\in E_n$ and $d_n(u,v)=k$ for some
$k\leq q$, we have the followings.
\begin{enumerate}
\renewcommand{\theenumi}{(\arabic{enumi})}
\item\label{c0} $E_n$ is an $H$-relation on $V$, $V_n\subseteq V_{n+1}$, $E_n\subseteq E_{n+1}$, $l_n\subseteq l_{n+1}$ and $d_n\subseteq d_{n+1}$.
\item\label{c1} $l_n(u,v)\in F^m_k$ and $1^{'}\in color_m(l_n(u,v))$ if and only if $u=v$.
\item\label{c2} If $(v,u)\in E_n$ then $d_n(v,u)\in\{k-1,k,k+1\}$.
\item\label{c3} $d_n(u,u)\in\{k-1,k,k+1\}$ and $d_n(v,v)\in\{k-1,k,k+1\}$.
\item\label{c4} If $\{(u,w),(w,v)\}\subseteq E_n$, then $\{d_n(u,w),d_n(w,v)\}\subseteq\{k-1,k,k+1\}$.
\item\label{c5} Suppose that $k\geq 1$, $(v,u)\in E_n$ and $d_n(v,u)=k-1$. Then either $d_n(u,u)=k$ and $d_n(v,v)=k-1$, or,  $d_n(u,u)=k-1$ and $d_n(v,v)=k$.
\item\label{c6} Suppose that $(v,u)\not\in E_n$ or $(v,u)\in E_n$ but $d_n(v,u)=k$. Then either $d_n(u,u)\not=k+1$ or $d_n(v,v)\not=k+1$.
\item\label{c7} If $\{(u,w),(w,v)\}\subseteq E_n$, then $\f{m}{WA}\models l_n(u,v)\cdot(l_n(u,w);l_n(w,v))\not=0$.
\item\label{c8} $RRA_H\models J(u,v):=\zeta(u,v)\cdot\zeta(u,u)\cdot\zeta(v,v)\not=0$. Where $\zeta(u,v)$, $\zeta(u,u)$ and $\zeta(v,v)$ are given as follows. If $(v,u)\in E_n$ then $\zeta(u,v)=l_n(u,v)\cdot l_n(v,u)\breve{ }$, otherwise $\zeta(u,v)=l_n(u,v)\cdot-\breve{1}$. If $(u,u)\in E_n$ then $\zeta(u,u)=l_n(u,u);l_n(u,v)$, otherwise $\zeta(u,u)=-(1^{'};l_n(u,v))$. If $(v,v)\in E_n$ then $\zeta(v,v)=l_n(u,v);l_n(v,v)$, otherwise $\zeta(v,v)=-(l_n(u,v);1^{'})$.
\end{enumerate}
Conditions (\ref{c1}), (\ref{c2}), (\ref{c3}) and (\ref{c4}) are used to show that the edges carry labels that don't interrupt the desired consistency in the sense of remark \ref{fayroz}. We construct the graph $G^{\tau}$ inductively, conditions (\ref{c5}) and (\ref{c6}) are used for the induction step together with conditions (\ref{c7}) and (\ref{c8}) which allow us to give labels of the degrees we need in a consistent way with our purpose as we shall see.

For constructing $G_0$, pick two different nodes $u,v$. Define,
$V_0=\{u,v\}$ and $E_0=\{(u,v),(v,u)\}\cup\{(u,u):\tau\in
dom(L_m)\}\cup\{(v,v):\tau\in dom(R_m)\}$. Define, $d_0(u,v)=d_0^{\c}(u,u)=q$ and $d_0(v,u)=d_0(v,v)=q-1$. Remember that $\tau$ is satisfiable form, then there exists an algebra $\B\in RRA_H$, an evaluation $\iota$ and $(r,s)$ in the unit of $\B$ such that $(\B,\iota,(r,s))\models \tau$. Define,
$l_0^{\c}(u,v)=\tau$, $l_0^{\c}(v,u)=f_m(\tau)$,
$l_0^{\c}(u,u)=D_q^{\B,\iota}(r,s)$ only if $\tau\in dom(L_m)$ and
$l_0^{\c}(v,v)=R_m(\tau)$ only if $\tau\in dom(R_m)$. We note
that $f_m(\tau)$ exists because $RRA_H\models 0\ne\tau\le t$.

We have to check that $G_0$ satisfies the consistency conditions.
Here, only conditions (\ref{c7}), (\ref{c8}) need a little thought:
they are satisfied because $\tau$ is a nonzero normal form and the
labels are given by the algebra $\B$ and the evaluation $\iota$. We
need to extend our piece of $G^{\c}$ in a way that guarantees that
$(u,v)$ satisfies $\tau$ at the end of the construction. Whence, we
need to add decompositions for all the edges according to the
information given by $sub_m$ of their labels. We don't care about
any other information because our strategy goes as follows.
Simultaneously with constructing a new edge, we add its converse and
its loops according to the information given by $f_m(\tau)$,
$L_m(\tau)$ and $R_m(\tau)$.

More generally, let $n\in\omega$ and suppose that
$G_n=(V_n,E_n),l_n,d_n$ have been constructed. In the round $n+1$,
let $W_n=\{e\in E_{n}:d_n(e)\geq 1\text{, and, }e\not\in
E_{n-1}\text{ if }n\geq 1\}$, i.e, the set of the edges that we have
to consider the information given by the black couples in their
$sub_m's$. Let $Y_n=\{(a,b)\in W_n:(b,a)\not\in E_n\text{, or,
}(b,a)\in E_n\text{ and }d_n(a,b)\geq d_n(b,a)\}$. In fact, for
every edge $(a,b)\in W_n\setminus Y_n$, we don't need to consider
the black forms $(\sigma_1,\sigma_2)\in sub_m(l_n(a,b))\cap
{^2dom(f_m)}$. Indeed, $(b,a)\in Y_n$ and we add edges according to
the information given by $sub_m(l_n(b,a))$. The converses of some of
these newly added edges are also added and they are enough to carry
the information given by $sub_m(l_n(a,b))$. Let $U$ be an infinite
set disjoint from $V_n$. For every $e\in Y_n$, create an injective
function
$$g_e:\{(\sigma_1,\sigma_2)\in sub_m(l_n(e)):\sigma_1,\sigma_2\text{ are both black}\}\rightarrow U,$$
and for every $e\in W_n\setminus Y_n$, create an injective function
$$g_e:\{(\sigma_1,\sigma_2)\in sub_m(l_n(e))\setminus{^2dom(f_m)}:\sigma_1,\sigma_2\text{ are both black}\}\rightarrow U,$$
such that the ranges of the functions $g_e$'s are pairwise disjoint. Let $e=(a,b)\in W_n$ and let $w\in Rng(g_e)$. Then there exists a pair of black forms $(\sigma_1,\sigma_2)\in sub_m(l_n(e))\cap
dom(g_e)$ such that $g_e((\sigma_1,\sigma_2))=w$. Remember that
$G_n$ satisfies the consistency condition (\ref{c7}). Therefore, there
exists an algebra $\B\in RRA_H$, an evaluation $\iota$ and $(r,s)$
in the unit of $\B$ such that $(\B,\iota,(r,s))\models l_n^{\c}(e)$.
Since both $\sigma_1,\sigma_2$ are black, then there exists $p$ in
the base of $\B$ different from $r$ and $s$ such that
$(r,p),(p,s)$ are in the unit of $\B$, $(\B,\iota,(r,p))\models\sigma_1$ and
$(\B,\iota,(p,s))\models\sigma_2$. If $S\in H$, then $(p,r),(s,p)$
are both in the unit of $\B$. If $R\in H$, then $(p,p)$ is in the
unit of $\B$. If $\sigma_1\in dom(f_m)$, then $(p,r)$ is in the unit
of $\B$ and $(\B,\iota,(p,r))\models f_m(\sigma_1)$. If $\sigma_2\in
dom(f_m)$, then $(s,p)$ is in the unit of $\B$ and
$(\B,\iota,(s,p))\models f_m(\sigma_2)$. If $\sigma_1\in dom(R_m)$
(hence $\sigma_2\in dom(L_m)$), then $(p,p)$ is in the unit of $\B$
and $(\B,\iota,(p,p))\models R_m(\sigma_1)(=L_m(\sigma_2))$. Suppose
that $d_n(e)=k$, for some $k\geq 1$.
\begin{description}
\item[Case 1] Suppose that $(b,a)\in E_n$ and $d_n(b,a)=k-1$. We define $E_w$ as follows.
\begin{eqnarray*}
    E_w:&=&\{(a,w),(w,b)\}\\
    && \cup\{(w,a):\sigma_1\in dom(f_m)\text{ or }S\in H\}\\
    && \cup\{(b,w):\sigma_2\in dom(f_m)\text{ or }S\in H\}\\
    && \cup\{(w,w):\sigma_1\in dom(L_m)\text{ or }R\in H\}.
\end{eqnarray*}
Suppose that $d_n(a,a)=k$ and $d_n(b,b)=k-1$. Define, $$l_w(a,w)=\sigma_1,\text{ }l_w(w,b)=\sigma_2,$$
\begin{eqnarray*}
l_w(w,a)&=&\begin{cases}
      D_{k-1}^{\B,\iota}(p,r) & \sigma_1\in dom(f_m)\text{ or }S\in H \\
      \text{not defined} & otherwise,
   \end{cases}\\
l_w(b,w)&=&\begin{cases}
      f_m(\sigma_2) & \sigma_2\in dom(f_m)\\
      D_{0}^{\B,\iota}(s,p) & \sigma_2\not\in dom(f_m)\text{ and }S\in H \\
      \text{not defined} & otherwise,
   \end{cases} and \\
   l_w(w,w)&=&\begin{cases}
      L_m(\sigma_1) & \sigma_1\in dom(L_m)\\
      D_{0}^{\B,\iota}(p,p) & \sigma_1\not\in dom(L_m)\text{ and }R\in H \\
      \text{not defined} & otherwise.
   \end{cases}
\end{eqnarray*}
For the other case when $d_n(a,a)=k-1$ and $d_n(b,b)=k$, we define $E_w$ and $l_w$ in a symmetric way to above. For the depths, define $d_w(w,w)=k-2$,if $k\geq 2$, and $d_w(w,w)=0$ otherwise. For every edge $e\in E_w$, $d_w(e)$ is defined to be the degree of the normal form $l_w(e)$.
\end{description}

\begin{figure}[!h]
\centering
\begin{tikzpicture}
\tikzset{edge/.style = {->,> = latex'}}
\node (a) at  (0,0) {$a$};
\node (f) at  (-0.3,-1) {$k$};
\node (b) at  (5,0) {$b$};
\node (g) at  (5.3,-1) {$k-1$};
\node (p) at (2.5,4) {$w$};
\node (n) at (2.5,5) {$k-2$};
\node (s) at (1.25,3) {$k-1$};
\node (k) at (3.75,3) {$k-1$};
\node (o) at (1.9,2) {$k-1$};
\node (r) at (3.1,2) {$k-2$};
\draw[edge] (a) to[bend left] (p);
\draw[edge] (2.45,3.8) to[bend left] (0.155,0.15);
\draw[edge] (4.855,0.15) to[bend left] (2.55,3.8);
\draw[edge] (p) to[bend left] (b);
\Loop[dist=1cm,dir=NO,labelstyle=above](p);
\draw[edge] (a)  to[bend left] (b);
\draw[edge] (b) to[bend left] (a);
\node[shape=circle] (h) at  (2.5,0.5) {$k$};
\node[shape=circle] (d) at  (2.5,-0.5) {$k-1$};
\Loop[dist=1cm,dir=SOEA,labelstyle=below right](b);
\Loop[dist=1cm,dir=SOWE,labelstyle=below left](a);
\end{tikzpicture}
\label{fig} \caption{Case 1}
\end{figure}

\begin{description}
\item[Case 2] Suppose that $(b,a)\not\in E_n$, or, $(b,a)\in E_n$ and $d_n(b,a)\geq k$. We have one of the following cases.
\begin{enumerate}[(i)]
\item Suppose that $d_n(a,a)\not=k+1$ and $d_n(b,b)\not=k+1$. Let
\begin{eqnarray*}
    E_w:&=&\{(a,w),(w,b)\}\\
    && \cup\{(w,a):\sigma_1\in dom(f_m)\text{ or }S\in H\}\\
    && \cup\{(b,w):\sigma_2\in dom(f_m)\text{ or }S\in H\}\\
    && \cup\{(w,w):\sigma_1\in dom(L_m)\text{ or }R\in H\}.
\end{eqnarray*}
We define $l_w$ as follows: $l_w(a,w)=\sigma_1,
l_w(w,b)=\sigma_2$,
\begin{eqnarray*}
l_w(w,a)&=&D_{k-1}^{\B,\iota}(p,r)\text{ if and only if }\sigma_1\in dom(f_m)\text{ or }S\in H,\\
l_w(b,w)&=&D_{k-1}^{\B,\iota}(s,p)\text{ if and only if }\sigma_2\in dom(f_m)\text{ or }S\in H, and,\\
l_w(w,w)&=&\begin{cases}
      L_m(\sigma_1) & \sigma_1\in dom(L_m)\\
      D_{0}^{\B,\iota}(p,p) & \sigma_1\not\in dom(L_m)\text{ and }R\in H \\
      \text{not defined} & otherwise.
   \end{cases}
\end{eqnarray*}
Define $d_w(w,w)=k-2$, if $k\geq 2$, and $d_w(w,w)=0$ otherwise. For every edge $e\in E_w$, define $d_w(e)$ is defined to be the degree of the normal form $l_w(e)$.
\item Suppose that $d_n(a,a)=k+1$, then $d_n(b,b)\not=k+1$. Let
\begin{eqnarray*}
    E_w:&=&\{(a,w),(w,b)\}\\
    && \cup\{(w,a):D_{k}^{\B,\iota}(p,r)\in dom(f_m)\}\\
    && \cup\{(b,w):\sigma_2\in dom(f_m)\text{ or }S\in H\}\\
    && \cup\{(w,w):D_{k}^{\B,\iota}(p,r)\in dom(L_m)\text{ or }R\in H\}.
\end{eqnarray*}
Define, $l_w(a,w)=D_{k}^{\B,\iota}(r,p)$, $l_w(w,b)=\sigma_2$,
\begin{eqnarray*}
l_w(w,a)&=&D_{k}^{\B,\iota}(p,r)\text{ if and only if }D_{k}^{\B,\iota}(r,p)\in dom(f_m)\\
l_w(b,w)&=&D_{k-1}^{\B,\iota}(s,p)\text{ if and only if }\sigma_2\in dom(f_m)\text{ or }S\in H, and,\\
l_w(w,w)&=&L_m(D_{k}^{\B,\iota}(r,p))\text{ if and only if }D_{k}^{\B,\iota}(r,p)\in dom(L_m).
\end{eqnarray*}
Similarly, for the case $d_n(b,b)=k+1$, we define $E_w$ and $l_w$ in the same spirt of the above item. Define $d_w(w,w)=k-1$ and, for every edge $e\in E_w$, $d_w(e)$ is defined as expected.
\end{enumerate}
\end{description}
Define the graph $G_{n+1}:=(V_{n+1},E_{n+1})$, the labeling $l_{n+1}$ and the depth $d_{n+1}$ as follows.
\begin{eqnarray*}
V_{n+1}=V_n\cup\bigcup\{Rng(g_e):e\in W_n\}, \text{ } \text{ } \text{ } E_{n+1}=E_{n}\cup\bigcup\{E_w:w\in V_{n+1}\setminus V_n\},\\
l_{n+1}=l_{n}\cup\bigcup\{l_w:w\in V_{n+1}\setminus V_n\}, \text{ } \text{ } \text{ } d_{n+1}=d_{n}\cup\bigcup\{d_w:w\in V_{n+1}\setminus V_n\}.
\end{eqnarray*}

Note that, the special choices of the labels and the depths of the new edges guarantee that $G_{n+1}$ is subjected to the consistency conditions listed above. In fact, we used the assumption that $G_n$ satisfies these conditions to get some algebras which satisfy some pieces of $G_n$ then we used these algebras to label the extra edges added to these pieces to get $G_{n+1}$ satisfying the required conditions. We continue building the graph $G^{\c}$ following the same argument by considering the information given by $sub_m$ of the labels of the edges. So $G^{\c}$ is constructed, basically, by knitting particular pieces of some members of $RRA_H$ and by assigning compatible depths caring remark \ref{fayroz}. This is what we targeted by the consistency conditions. We note that these choices of the depths are not the only possible choices, but with these choices we could reach our aim. Also, one may notice that the depths of the loops are the keys we use to follow remark \ref{fayroz}. We note that the resulting graph might be infinite independently from the choices of the depths.

Let $G^{\c}=(V^{\c},E^{\c})$, where $V^{\c}:=\bigcup\{V_n:n\in\omega\}$ and $E^{\c}:=\bigcup\{E_n:n\in\omega\}$. We still need the depths and the labels, let $d^{\c}=\bigcup\{d_n:n\in\omega\}$ and $l^{\c}=\bigcup\{l_n:n\in\omega\}$. Now, define $\a{G}^{\c}$ as the full $RRA_H$ algebra with unit $E^{\c}$,
$$\a{G}^{\c}=\langle\mathcal{P}(E^{\c}),\cap,\cup,\setminus,\emptyset,E^{\c},;^{[E^{\c}]},\Breve{\text{ }}^{[E^{\c}]}\rangle.$$
By the consistency conditions \ref{c0}, every $E_n$ is an
$H$-relation on $V_n$. Consequently, $\a{G}^{\c}\in RRA_H$ as
desired. Define an evaluation of the free variables
$x_0,\ldots,x_{m-1}$ as follows. For every $i<m$, define
$\iota^{\c}(x_i)=\{e\in E^{\c}:x_i\in color_m(l^{\c}(e))\}$. Now, we
need to prove the following proposition.
\paragraph{Step 2:} In
the next proposition, we prove that every edge in $E^{\tau}$
satisfies its label in $(\G,\iota^{\c})$. Therefore, in particular,
The unique edge $(u,v)$ satisfies $\tau$ in $(\G,\iota^{\c})$.
\begin{prop}\label{satisfy}For every edge $e\in E^{\c}$, $(\G,\iota^{\c},e)\models l^{\c}(e)$.
\end{prop}
\begin{proof}
We make use of the consistency conditions (\ref{c1})-(\ref{c8}). Let $e=(e_0,e_1)\in E^{\c}$. Then, by condition (\ref{c8}), we have $\f{m}{RRA_H}\models l^{\c}(e)\not=0$. For every $0\leq k\leq d^{\c}(e)$, let $tag_k^{\c}(e)$ be the unique term in $F_k^m$ such that $\f{m}{RRA_H}\models l^{\c}(e)\leq tag_k^{\c}(e)$. For every $d^{\c}(e)\leq k\leq q$, let $tag_k^{\c}(e):=l^{\c}(e)$. It suffices to prove that
\begin{equation}
(\forall e\in E^{\c}) \text{ } (\forall 0\leq k\leq q) \text{ } \text{ } \text{ } (\G,\iota^{\c},e)\models tag_k^{\c}(e). \tag{*}
\end{equation}
For this, we use induction on $k$. Condition (\ref{c1}) and
the special choice of $\iota^{\c}$ ensure that
$(\G,\iota^{\c},e)\models tag_0^{\c}(e)$, for every $e\in E^{\c}$.
Suppose that, for some $0\leq k\leq q-1$, $(\G,\iota^{\c},e)\models
tag^{\c}_k(e)$ for every $e\in E^{\c}$. Let $(e_0,e_1)\in E^{\c}$,
if $k\geq d^{\c}(e)$ then $(\G,\iota^{\c},e)\models
tag_{k+1}^{\c}(e)=tag_k^{\c}(e)=l^{\c}(e)$. So, we can assume that
$k<d^{\c}(e)$ (then $j:=d^{\c}(e)\geq 1$). Let
$\alpha\in{^{D_m}\{-1,1\}}$, $\beta\in{^{CF_k^m}\{-1,1\}}$ and
$\beta^{'}\in{^{CF_{j-1}^m}\{-1,1\}}$ be such that
$$tag_{k+1}^{\c}(e):=D_m^{\alpha}\cdot (CF_{k}^m)^{\beta}\text{ and
}\sigma:=D_m^{\alpha}\cdot (CF_{j-1}^m)^{\beta^{'}}.$$

By condition (\ref{c1}) and the special choice of the evaluation $\iota^{\c}$, we have \begin{equation}\label{0}
(\G,\iota^{\c},e)\models D_m^{\alpha}.
\end{equation}
Let $\gamma\in F_k^m$. Suppose that $\beta(\breve{\gamma})=1$. Then there exists $\gamma_1\in F_{j-1}^m$ such that $\f{m}{RRA_H}\models\gamma_1\leq\gamma$ and $\beta^{'}(\breve{\gamma}_1)=1$, i.e., $\f{m}{RRA_H}\models\sigma\cdot\breve{\gamma}_1\not=0$. By conditions (\ref{c8}) and (\ref{c2}), there exists $l\geq j-1$, $\gamma_2\in F_{l}^m$ and an edge such that $(e_1,e_0)\in E^{\c}$, $l^{\c}(e_1,e_0)=\gamma_2$ and $\f{m}{RRA_H}\models\sigma\cdot\breve{\gamma}_2\not=0$. Let $\gamma_3$ be the unique term in $F_{j-1}^m$ such that $\f{m}{RRA_H}\models\gamma_2\leq\gamma_3$. Hence, $\f{m}{RRA_H}\models\sigma\cdot\breve{\gamma}_3\not=0$ and, consequently, $\beta^{'}(\breve{\gamma}_3)=1$. Recall that $\beta^{'}(\breve{\gamma}_1)=1$, then $\f{m}{RRA_H}\models\sigma\cdot(\gamma_3\cdot\gamma_1)\breve{ }=\sigma\cdot\breve{\gamma}_3\cdot\breve{\gamma}_1\not=0$. By theorem \ref{andreka}, \ref{and2}, this happens only if $\gamma_1=\gamma_3$. Hence, $\f{m}{RRA_H}\models\gamma_2\leq\gamma_1$. By induction hypothesis, we have $(\G,\iota^{\c},(e_1,e_0))\models tag_k^{\c}((e_1,e_0))=\gamma$. Therefore, $(\G,\iota^{\c},e)\models\breve{\gamma}$. Conversely, suppose that $\beta(\breve{\gamma})=-1$. Assume toward a contradiction that there exist $l\geq j-1$, $\gamma_1\in F_{l}^m$ such that $(e_1,e_0)\in E^{\c}$, $l^{\c}(e_1,e_0)=\gamma_1$ and $(\G,\iota^{\c},(e_1,e_0))\models\gamma$. By the induction hypothesis and theorem \ref{andreka}, \ref{and2}, we should have $tag_k^{\c}((e_1,e_0))=\gamma$. By condition (\ref{c8}), $\f{m}{RRA_H}\models\sigma\cdot\breve{\gamma}_1\not=0$. Hence, $\f{m}{RRA_H}\models tag_{k+1}^{\c}(e)\cdot\breve{\gamma}\not=0$, which makes a contradiction. Therefore,
\begin{equation}\label{1}
(\G,\iota^{\c},e)\models\breve{\gamma}\iff\beta(\breve{\gamma})=1.
\end{equation}
Let $\lambda\in F_k^m$ be a white term and $\gamma\in F_k^m$ be a black term. Suppose that $\beta(\lambda;\gamma)=1$. Then there exist a white $\lambda_1\in F_{j-1}^m$ and a black $\gamma_1\in F_{j-1}^m$ such that $\f{m}{RRA_H}\models\lambda_1\leq\lambda$, $\f{m}{RRA_H}\models\gamma_1\leq\gamma$ and $\beta^{'}(\lambda_1;\gamma_1)=1$. Hence, $\f{m}{RRA_H}\models\sigma\cdot(\lambda_1;\gamma_1)\not=0$. But $\f{m}{RRA_H}\models\lambda_1;\gamma_1\leq\gamma_1$, then $\f{m}{RRA_H}\models\sigma\leq\gamma_1$ and $\f{m}{RRA_H}\models\sigma\cdot(\lambda_1;\sigma)\not=0$. By conditions (\ref{c8}) and (\ref{c3}) there exist $l\geq j-1\geq k$ and $\lambda_2\in F^m_l$ such that $(e_0,e_0)\in E^{\c}$, $l^{\c}(e_0,e_0)=\lambda_2$ and $\f{m}{RRA_H}\models\lambda_2;\sigma\not=0$. Let $\lambda_3$ be the unique term in $F_{j-1}^m$ such that $\f{m}{RRA_H}\models\lambda_2\leq\lambda_3$. Hence, $\f{m}{RRA_H}\models\lambda_3;\gamma_1\not=0$ and $\beta^{'}(\lambda_3;\gamma_1)=1$. Recall that $\beta^{'}(\lambda_1;\gamma_1)=1$, then $\f{m}{RRA_H}\models(\lambda_1;\gamma_1)\cdot(\lambda_3;\gamma_1)=(\lambda_1\cdot\lambda_3)\cdot\gamma_1\not=0$.
By theorem \ref{andreka}, \ref{and2}, this means that $\lambda_1=\lambda_3$ and $\f{m}{RRA_H}\models\lambda_2\leq\lambda_1$. By induction, we have $(\G,\iota^{\c},((e_0,e_0)))\models tag_k^{\c}((e_0,e_0))=\lambda$ and $(\G,\iota^{\c},e)\models\gamma$. Therefore, $(\G,\iota^{\c},e)\models\lambda;\gamma$. Conversely, suppose that $\beta(\lambda;\gamma)=-1$ and $(\G,\iota^{\c},e)\models\gamma$. Assume toward a contradiction that there exist $l\geq j-1$, $\lambda_1\in F_l^m$ such that $(e_0,e_0)\in E^{\c}$, $l^{\c}(e_0,e_0)=\lambda_1$ and $(\G,\iota^{\c},(e_0,e_0))\models\lambda$. By the induction hypothesis and theorem \ref{andreka}, \ref{and2}, we should have $tag_k^{\c}((e_0,e_0))=\lambda$. By condition (\ref{c8}), $\f{m}{RRA_H}\models\sigma\cdot(\lambda_1;\sigma)\not=0$. Hence, $\f{m}{RRA_H}\models tag_{k+1}^{\c}(e)\cdot(\lambda;\gamma)\not=0$, which makes a contradiction. Therefore,
\begin{equation}\label{2}
(\G,\iota^{\c},e)\models\lambda;\gamma\iff\beta(\lambda;\gamma)=1.
\end{equation}
Similarly, \begin{equation}\label{3}
(\G,\iota^{\c},e)\models\gamma;\lambda\iff\beta(\gamma;\lambda)=1.
\end{equation}
Let $\gamma_1,\gamma_2\in F_k^m$ be black terms. Suppose that $\beta(\gamma_1;\gamma_2)=1$, then there exists $\gamma_1^{'},\gamma_2^{'}\in F_{j-1}^m$ such that $\beta^{'}(\gamma_1^{'};\gamma_2^{'})=1$, $\f{m}{RRA_H}\models\gamma_1^{'}\leq\gamma_1$ and $\f{m}{RRA_H}\models\gamma_2^{'}\leq\gamma_2$. By the construction and condition (\ref{c4}), there exist a node $w\in g_e$, $j_1,j_2\geq l-1$ and $\gamma_1^{''}\in F_{j_1}^m$ and $\gamma_2^{''}\in F_{j_2}^m$ such that $\f{m}{RRA_H}\models\gamma_1^{''}\leq\gamma_1^{'}$, $\f{m}{RRA_H}\models\gamma_2^{''}\leq\gamma_2^{'}$, both $(e_0,w),(w,e_1)$ are edges in $E^{\c}_{n+1}$, $l^{\c}(e_0,w)=\gamma_1^{''}$ and $l^{\c}(w,e_1)=\gamma_2^{''}$. By induction, we have $(\G,\iota^{\c},(e_0,w))\models tag_k^{\c}((e_0,w))=\gamma_1$ and $(\G,\iota^{\c},(w,e_1))\models tag_k^{\c}((w,e_1))=\gamma_2$. Therefore, $(\G,\iota^{\c},e)\models\gamma_1;\gamma_2$. Conversely, Suppose that $\beta(\gamma_1;\gamma_2)=-1$ and assume toward a contradiction that there exists a node $w\in V^{\c}$, $j_1,j_2\geq l-1$ and $\gamma_1^{'}\in F_{j_1}^m$ and $\gamma_2^{'}\in F_{j_2}^m$ such that both $(e_0,w),(w,e_1)\in E^{\c}$, $l^{\c}(e_0,w)=\gamma_1^{'}$, $l^{\c}(w,e_1)=\gamma_2^{'}$, $(\G,\iota^{\c},(e_0,w))\models\gamma_1$ and $(\G,\iota^{\c},(w,e_1))\models\gamma_2$. By a similar argument, using condition (\ref{c7}), one can show a contradiction. Therefore,
\begin{equation}\label{4}
(\G,\iota^{c},e)\models\gamma_1;\gamma_2\iff\beta(\gamma_1;\gamma_2)=1.
\end{equation}
By (\ref{0}), (\ref{1}), (\ref{2}), (\ref{3}) and (\ref{4}), it follows that $(\G,\iota^{\c},e)\models tag_{k+1}^{\c}(e)$, as desired.
\end{proof}

The algebra $\a{G}^{\c}$ and the evaluation $\iota^{\c}$ have been
constructed and were shown to witness satisfiability of $\tau$. Now,
we carry on with our plan and move to the next step.

\paragraph{Step 3:} We need to find a sequence (zigzag!) of
special edges.
\begin{definition}An edge $(a,b)\in E^{\c}$ is said to be
\begin{enumerate}[-]
\item[-] a useful edge if $(b,a)\in E^{\c}$, $d^{\c}(b,a)<d(a,b)$, and either $(a,b)$ is the unique edge in $E^{\c}$ whose depth is $q$ and $a\not=b$, or, there exists a unique node $w\in V^{\c}$ such that $\{(a,w),(w,b),(w,a),(b,w)\}\subseteq E^{\c}$ and $d^{\c}(a,w)$, $d^{\c}(w,b)$, $d^{\c}(w,a)$, $d^{\c}(b,w)$ are all greater than or equal to $d^{\c}(a,b)$.
\item[-] a side edge if $d^{\c}(a,b)=0$, $(b,a)\in E^{\c}$ if and only if $S\in H$, $\{(a,a),(b,b)\}\subseteq E^{\c}$ if and only if $R\in H$ and $d(a,a)=d(b,b)=0$.
\end{enumerate}
\end{definition}
By the above definition, the unique edge $(u_q,v_q):=(u,v)\in E_0$ with label $\tau$ and depth $q$ is useful edge. Recall that $\f{m}{RRA_H}\models\tau\leq t$. Therefore, there exists a node $w_q\in V^{\c}$ such that $\{(u_q,w_q),(w_q,u_q),(w_q,v_q),(v_q,w_q)\}\subseteq E_1$. By the construction (Case 1), there exists an
edge, $e_{q-1}=(u_{q-1},v_{q-1})\in\{(u_q,w_q),(w_q,u_q)\}\subseteq E_1$ that is useful and $d^{\c}(e_{q-1})=q-1$. Continuing in the same manner, we get a sequence of useful edges $e_q=(u_q,v_q),\ldots,e_1=(u_1,v_1)\in E^{\c}$, a side edge $e_0=(u_0,v_0)$ and a sequence of nodes $w_q,\ldots,w_1\in V^{\c}$ such that
\begin{enumerate}[(a)]
\item $l^{\c}(e_q)=\tau$ and $d^{\c}(e_j)=j$ for every $0\leq j\leq q$.
\item $e_k\in\{(u_{k+1},w_{k+1}),(w_{k+1},v_{k+1})\}$ for every $0\leq k<q$.
\item\label{a7mos} For every $0<k\leq q$ and every node $y\in V^{\c}\setminus\{w_{k}\}$, if $\f{m}{RRA_H}\models l^{\c}(u_{k},y)\cdot l^{\c}(u_{k},w_{k})\not=0$ and $\f{m}{RRA_H}\models l^{\c}(y,v_{k})\cdot l^{\c}(w_{k},v_{k})\not=0$ then $z$ is the unique node for the useful edge $e_k$ mentioned in the above definition.
\end{enumerate}

\paragraph{Step 4:} We note that the selection of the sequence
$e_q,\ldots,e_0$ is not unique. The idea of showing that $\tau$ is
not an atom in $\f{m}{RRA_H}$ is as follows. We extend $G^{\c}$ to
$G^{\c}_{+}$ as follows. Recall that we have either $m\not=0$ or
$H\not=\{R,S\}$.
\begin{description}
\item[Suppose that $m\geq 1$] Pick a brand new node $h$ and let $V_{+}^{\c}=V^{\c}\cup\{h\}$.
Define $E_{+}^{\c}$ as follows.
\begin{eqnarray*}
E_{+}^{\c}&=&E^{\c}\cup\{(u_0,h),(h,v_0)\}\cup\{(h,h):R\in H\}\\
&& \cup\{(h,u_0):S\in H\}\cup\{(v_0,h):S\in H\}
\end{eqnarray*}
We extend the labels as follows. For every $e\in E^{\c}$, let $l_{+}^{\c}(e)=l^{\c}(e)$. If there is
no $z\in V^{\c}$ with $\{(u_0,z),(z,v_0)\}\subseteq E^{\c}$ (in
other words, if $S\not\in H$), define
$$l^{\c}_{+}(u_0,h)=l_{+}^{\c}(h,u_1)=0^{'}\cdot\prod_{i\in
m}-x_i.$$ If there is $z\in V^{\c}$ with
$\{(u_0,z),(z,v_0)\}\subseteq E^{\c}$, then $z$ is unique because $(u_0,v_0)$ is a side edge of $G^{\c}$. Since $m\geq 1$, then there exists $\gamma_1,\gamma_2\in
F_0^m$ such that $color_m(\gamma_1)\not=color_m(l^{\c}(u_0,z))$ and
$color_m(\gamma_{2})\not=color_m(l^{\c}(z,u_1))$.
Define,
$$l^{\c}_{+}(u_0,h)=\gamma_1\text{ and }
l_{+}^{\c}(h,u_1)=\gamma_2.$$ Define,
 $$l_{+}^{\c}(h,u_0)=0^{'}\cdot\prod_{i\in m}-x_i\text{ if and only if }S\in H,$$
    $$l_{+}^{\c}(u_1,h)=0^{'}\cdot\prod_{i\in m}-x_i\text{ if and only if }S\in H$$ and $$l_{+}^{\c}(h,h)=1^{'}\cdot\prod_{i\in m}-x_i\text{ if and only if }R\in H.$$
 The depths are extended in the natural way, $d_{+}^{\c}(e)=d^{\c}(e)$, for every $e\in dom(d^{\c})$, $d_{+}^{\c}(h,h)=0$ and $d_{+}^{\c}(e)=d$ if and only if $l_{+}^{\c}(e)\in F_d^m$, for every $e\in E_{+}^{\c}\setminus E^{\c}$ and every $d\leq q$.
\item[Suppose $m=0$ and $S\not\in H$] Hence, extend $G^{\c}$ by adding the converse of $e_0$. Let $V^{\c}_{+}:=V^{\c}$ and $E_{+}^{\c}=E^{\c}\cup\{(v_0,u_0)\}$. Define $l_{+}^{\c}(e):=l^{\c}(e)$, for every $e\in E^{\c}$, and $l_{+}^{\c}(v_0,u_0)=0^{'}\cdot\prod_{i\in m}-x_i$. The depths are given by $d_{+}^{\c}=d^{\c}\cup\{((v_0,u_0),0)\}$.
\item[Suppose $m=0$ and $R\not\in H$] We extend $G^{\c}$ by adding the loop of $(w_1,w_1)$. Let $V^{\c}_{+}:=V^{\c}$ and $E_{+}^{\c}=E^{\c}\cup\{(w_1,w_1)\}$. Define $l_{+}^{\c}(e):=l^{\c}(e)$, for every $e\in E^{\c}$ and $l_{+}^{\c}(w_1,w_1)=1^{'}\cdot\prod_{i\in m}-x_i$.
The depths remain as they were, i.e., $d^{\c}_{+}=d^{\c}$.
\end{description}
Let $\a{G}^{\c}_+=\langle\mathcal{P}(E^{\c}_+),\cap,\cup,\setminus,\emptyset,E^{\c}_+,;^{[E^{\c}_+]},\Breve{\text{ }}^{[E_+^{\c}]}\rangle$. Clearly $\a{G}^{\c}_+\in RRA_H$. Define the evaluation  $\iota^{\c}_{+}:\{x_0,\ldots,x_{m-1}\}\rightarrow\mathcal{P}(E^{\c}_{+})$ as follows. For every $i\in m$, $\iota^{\c}_{+}(x_i)=\{e\in E^{\c}_{+}: x_i\in color_m(l^{\c}_{+}(e))\}$. The same argument used in proposition \ref{satisfy} can be used to prove the following.
\begin{equation}\label{satisfy2}
(\forall e\in E^{\c}_{+}) \text{ } \text{ } \text{ } (\a{G}^{\c}_{+},\iota^{\c}_{+},e)\models l_{+}^{\c}(e).
\end{equation}

\paragraph{Step 5:}
Recall the useful edges $e_q, \dots, e_1$ and the side edge $e_0$. For every $0\leq j\leq q$, let
$son(e_j)$ and $daughter(e_j)$ be the unique terms in $F_{j+1}^m$
such that \begin{center}$(\a{G}^{\c},\iota^{\c},e_j)\models son(e_j)$
and $(\a{G}^{\c}_+,\iota^{\c}_{+},e_j)\models
daughter(e_j)$.
\end{center}
Hence, for every $0\leq j\leq q$, $son(e_j)$ and $daughter(e_j)$ are
non-zero forms and, by lemma \ref{satisfy} and equation
(\ref{satisfy2}) above, each of which is below the label of $e_j$ in
$\f{m}{RRA_H}$. Now, we are ready to show that $\tau$ is not an atom
in $\f{m}{RRA_H}$. We use the conditions in the definition of the
useful edges to show the following.
\begin{prop}
For every $0\leq j\leq q$, $$\f{m}{RRA_H}\models son(e_j)\cdot daughter(e_j)=0.$$
\end{prop}
\begin{proof}
By induction on $j$. By the special construction of $G_{+}^{\c}$, it is clear that $\f{m}{RRA_H}\models son(e_0)\cdot daughter(e_0)=0$. Suppose that, for some $j\leq q-1$, $\f{m}{RRA_H}\models son(e_j)\cdot daughter(e_j)=0$. We may assume that $e_j=(u_{j+1},w_{j+1})$. Let $\sigma=D^{\a{G}^{\c},\iota^{\c}}_{j+1}((w_{j+1},v_{j+1}))$. By condition (\ref{a7mos}) and without loss of generality, we may assume that there is NO node $y\in V^{\c}\setminus\{w_{j+1}\}$ with $\{(u_{j+1},y),(y,v_{j+1})\}\subseteq E^{\c}$ and
\begin{eqnarray*}
(\a{G}^{\c},\iota^{\c},(u_{j+1},y))\models son(e_j), &\text{ }&
(\a{G}^{\c},\iota^{\c},(y,v_{j+1}))\models \sigma,\\
(\a{G}^{\c}_+,\iota^{\c}_+,(u_{j+1},y))\models son(e_j)
&\text{ }& (\a{G}^{\c}_+,\iota^{\c}_+,(y,v_{j+1}))\models \sigma.
\end{eqnarray*}
Hence, $$(\a{G}^{\c},\iota^{\c},(u_{j+1},v_{j+1}))\models son(e_j);\sigma\text{ } \text{ but } \text{ }(\a{G}^{\c}_+,\iota^{\c}_+,(u_{j+1},v_{j+1}))\models -(son(e_j);\sigma).$$
Therefore, $\f{m}{RRA_H}\models son(e_{j+1})\cdot daughter(e_{j+1})=0$, as desired.
\end{proof}

In particular, $\f{m}{RRA_H}\models son(e_{q})\cdot
daughter(e_{q})=0$ and, hence, $\tau$ is not an atom in the free
algebra $\f{m}{RRA_H}$. Recall that $\tau$ was an arbitrary normal
form with $\f{m}{RRA_H}\models0\not=\tau\leq t$. Hence, proving that
$\tau$ is not an atom in $\f{m}{RRA_H}$ yields to the following
proposition.
\begin{prop}\label{finish}Suppose that $m\not=0$ or $H\not=\{R,S\}$. There is no atom in the free algebra $\f{m}{RRA_H}$ that is below $t$. Therefore, $\f{m}{RRA_H}$ is not atomic
\end{prop}
The remaining direction of theorem \ref{didit} follows from
proposition \ref{finish}. It would be nice to list all the atoms of
the non atomic $\f{m}{RRA_H}$'s. However, we believe that the free
algebra $\f{m}{RRA_H}$ contains finitely many atoms if and only if
$S\in H$.

\paragraph{Acknowledgment:}
The research to this article was sponsored by Central European
University Foundation, Budapest (CEUBPF). The theses explained
herein are representing the own ideas of the author, but not
necessarily reflect the opinion of CEUBPF.

The research to this article was carried out at University College
London, London (UCL), during a research visit program under the
supervision of professor Robin Hirsch. It is also part of the
author's PhD thesis under the supervision of Professors Hajnal
Andr\'eka and Istv\'an N\'emeti.

The author should like to thank Professors Hajnal Andr\'eka and
Istv\'an N\'emeti for reading all the drafts of the present paper
and for their valuable comments and ideas. The author is thankful
for Professors Robin Hirsch and Szabolcs Mikul\'as for the frequent
meetings to discuss the ideas used in the presented results.

\bibliographystyle{plainnat}
\bibliography{RSL}
\vspace*{10pt}

\end{document}